\documentclass[12pt]{article}

\usepackage[colorlinks, linkcolor=blue, citecolor=blue]{hyperref}           
\usepackage{color}
\usepackage{graphicx,subfigure,amsmath,amssymb,amsfonts,bm,epsfig,epsf,url,dsfont}
\usepackage{booktabs}
\usepackage{amsthm,mathrsfs}

\newtheorem{thm}{Theorem}[section]
\newtheorem{lem}[thm]{Lemma}
\newtheorem{assum}[thm]{Assumption}
\newtheorem{proposition}[thm]{Proposition}

\usepackage{tikz}
\usepackage{bbm}
\usepackage[small,bf]{caption}
\usepackage[top=1in,bottom=1in,left=1in,right=1in]{geometry}
\usepackage{fancybox}
\usepackage{algorithm}
\usepackage{verbatim}
\usepackage[titletoc,toc]{appendix}

\usepackage{algpseudocode}

\usepackage{mathtools}

\usepackage{scalerel,stackengine}
\stackMath
\newcommand\reallywidehat[1]{%
\savestack{\tmpbox}{\stretchto{%
  \scaleto{%
    \scalerel*[\widthof{\ensuremath{#1}}]{\kern-.6pt\bigwedge\kern-.6pt}%
    {\rule[-\textheight/2]{1ex}{\textheight}}
  }{\textheight}%
}{0.5ex}}%
\stackon[1pt]{#1}{\tmpbox}%
}

\makeatletter
\newcommand*{\rom}[1]{\expandafter\@slowromancap\romannumeral #1@}
\makeatother

\DeclareMathOperator*{\argmin}{arg\,min}

\usepackage{xspace}

\numberwithin{equation}{section}

\allowdisplaybreaks

\begin{document}
\title{Provable method of moments for cryo-EM}

\makeatletter
\renewcommand{\thefootnote}{\arabic{footnote}}
\renewcommand{\@fnsymbol}[1]{\number#1}
\makeatother

\author{
Amnon Balanov\thanks{\scriptsize{Department of Electrical and Computer Engineering, Technion -- Israel Institute of Technology, Haifa, Israel.}},
Tamir Bendory\footnotemark[1],
Dan Edidin\thanks{\scriptsize{Department of Mathematics, University of Missouri, Columbia, MO, USA.}},
Josh Katz\thanks{\scriptsize{The MITRE Corporation, Bedford, MA, USA.}},\\
Shay Kreymer\thanks{\scriptsize{School of Electrical and Computer Engineering, Tel Aviv University, Tel Aviv Israel}},
Nir Sharon\thanks{\scriptsize{School of Mathematical Sciences, Tel Aviv University, Tel Aviv, Israel.}}
}

\maketitle

\begin{abstract}
We study the recovery of a three-dimensional volume in single-particle cryo-electron microscopy (cryo-EM) from noisy tomographic projections acquired at unknown viewing rotations, drawn independently and uniformly from $\mathrm{SO}(3)$. We show that, in a spherical-shell representation satisfying an explicit shell-separation condition, the first three moments generically determine the underlying real volume up to its intrinsic global $\mathrm{O}(3)$ ambiguity. The proof is constructive and yields a frequency-marching procedure that successively recovers the volume.
Consequently, consistent recovery in the high-noise regime has sample complexity $n=\omega(\sigma^6)$, where $n$ is the number of observations and $\sigma^2$ is the noise variance. Numerical experiments validate the constructive recovery scheme under finite sampling and additive noise, and indicate that direct third-moment optimization can succeed beyond the structural assumptions required by the constructive proof.
\end{abstract}

\renewcommand{\thefootnote}{$\ast$}
\footnotetext{\scriptsize{Affiliation with the MITRE Corporation is for identification purposes only and is not intended to convey or imply MITRE’s concurrence with, or support for, the positions, opinions, or viewpoints expressed by the
author. \copyright 2026 The MITRE Corporation. ALL RIGHTS RESERVED.}}
\section{Introduction}

Single-particle cryo-electron microscopy (cryo-EM) is one of the leading techniques for determining the three-dimensional structure of macromolecules from noisy two-dimensional projection images; see, for example,~\cite{kuhlbrandt2014resolution, bai2015cryo, bendory2020single,singer2020computational}. In the mathematical model considered here, the unknown volume $f:\mathbb{R}^3\to\mathbb{R}$ is observed through
\begin{align}
    y_i = P(R_i\cdot f) + \varepsilon_i,
    \qquad i=1,\dots,n,
    \label{eq:cryoEM}
\end{align}
where $P$ denotes the tomographic projection operator, $R_i\stackrel{\mathrm{i.i.d.}}{\sim}\operatorname{Haar}(\mathrm{SO}(3))$ are unknown viewing rotations, and $\varepsilon_i$ are independent Gaussian noise vectors with variance $\sigma^2$. We assume that $f$ belongs to a finite-dimensional space $V$ of band-limited functions on $\mathbb{R}^3$; the precise shell-model representation is introduced in Section~\ref{subsec:spherical-harmonics}. Because the observation model is invariant under global rotations and handedness reversal, $f$ can be identified only up to its global $\mathrm{O}(3)$-orbit; see Lemma~\ref{lem:reflection-invariance}.

A classical route to cryo-EM reconstruction goes back to Kam~\cite{kam1980reconstruction}, who observed that when viewing directions are uniformly distributed, low-order moments of the projection images encode rotationally averaged information about the underlying three-dimensional volume. From a modern perspective, this is a method-of-moments formulation of cryo-EM~\cite{bhamre2015orthogonal, sharon2020method,bendory2020single}.

From a mathematical perspective, the cryo-EM problem belongs to the class of projected orbit-recovery problems, in which an unknown signal is acted on by a latent element of a group and is then observed through a linear operator~\cite{bandeira2023estimation,balanov2026projected}. A central principle in this framework is that, in the high-noise regime, the sample-complexity exponent is governed by the lowest-order moment that identifies the signal orbit~\cite{perry2019sample,bandeira2023estimation,fan2024maximum,edidin2025orbit}. More precisely, if moments up to order $d-1$ fail to identify the orbit while the $d$th moment does, then the characteristic sample-complexity scale is $\sigma^{2d}$ as $\sigma\to\infty$, and consistent recovery requires $n/\sigma^{2d}\to\infty$. Since our main focus is precisely this high-noise regime, the central question is therefore whether the first three moments of the projected cryo-EM observations contain enough information to identify the underlying volume up to its global $\mathrm{O}(3)$ ambiguity, and whether this identification can be made constructively.

\paragraph{Frequency marching and identifiability.}
A recurring constructive approach in moment-based orbit recovery is \emph{frequency marching}: one first recovers a low-frequency or low-degree component, fixes the corresponding global group ambiguity, and then uses higher-order invariants that couple already recovered components with the next unknown component to proceed successively through frequency or harmonic degree. Variants of this principle appear in several orbit-recovery problems. In classical and dihedral multi-reference alignment, low-order moments can be used to recover the signal progressively across Fourier frequencies~\cite{perry2019sample,bandeira2023estimation,edidin2026reflection}. More recently, a constructive frequency-marching procedure was developed for the unprojected $\mathrm{SO}(3)$ problem, proving that degree-three invariants globally determine a generic band-limited spherical signal up to its orbit~\cite{bendory2025orbit,bendory2026provable}. A related strategy was subsequently adapted to projected MRA, where the first three moments still generically determine the appropriate dihedral orbit despite the additional linear projection~\cite{balanov2026projected}. 

For cryo-EM, a complementary line of work has addressed identifiability without relying on a frequency-marching construction. A generic Jacobian-rank analysis of the moment map showed that moments up to order three provide local generic identifiability in both the unprojected $\mathrm{SO}(3)$ model and the projected cryo-EM model~\cite{fan2024maximum}. This establishes that third-order moments generically capture the \emph{local degrees} of freedom of the orbit, but does not show that they \emph{globally} separate generic volumes or provide a constructive recovery procedure.

\paragraph{Main results.}
The preceding results leave open the corresponding global recovery question for the moments of single-particle cryo-EM. Under uniformly distributed viewing directions, the second moment determines 
the Gram matrices of each harmonic block and therefore leaves an independent orthogonal ambiguity for each block~\cite{kam1980reconstruction,bhamre2015orthogonal,bendory2024sample}. Although third-order invariants resolve the analogous ambiguities in the \emph{unprojected} $\mathrm{SO}(3)$ problem~\cite{bendory2025orbit}, after tomographic projection only certain combinations of these invariants remain observable (see Figure~\ref{fig:1}(a) and Section~\ref{subsec:projection_mixing} for further details). It is therefore not a priori clear whether the third projected moment retains enough information to resolve the blockwise ambiguities and identify the volume up to its intrinsic global $\mathrm{O}(3)$ symmetry.

In this paper, we show that, for uniformly distributed viewing directions and under an explicit finite-dimensional spherical-shell condition, the first three moments generically determine the underlying real volume up to its global $\mathrm{O}(3)$ orbit. The proof is constructive and yields the \emph{bidirectional frequency-marching} procedure illustrated in Figure~\ref{fig:1}(c)-(d): a forward march recovers isolating outer-shell coefficients in increasing harmonic degree, followed by a reverse backfill of the remaining inner-shell coefficients. Since the second projected moment does not generically identify the volume whereas the third does under our assumptions, the corresponding high-noise sample-complexity scale is $\sigma^6$, and consistent recovery requires $n=\omega(\sigma^6)$. We complement the theory with numerical experiments that validate the constructive recovery scheme from finite-sample moment estimates and investigate direct third-moment optimization beyond the structural assumptions required by the proof.

\begin{figure*}[t]
    \centering
    \includegraphics[width=\textwidth]{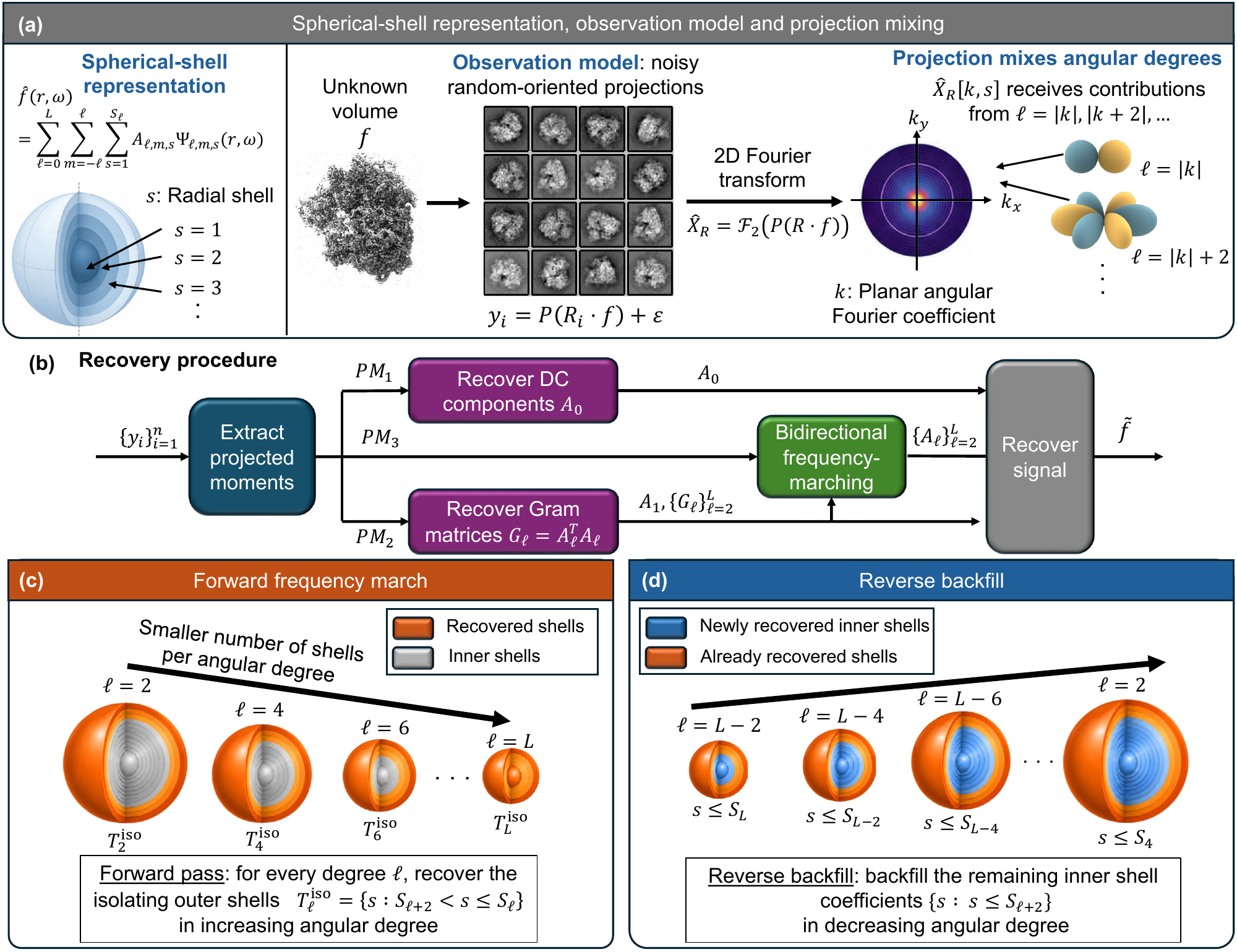}
    \caption{\textbf{Overview of the model and constructive recovery.} 
    (a) The volume is represented on spherical Fourier shells, while tomographic projection mixes harmonic invariants of matching parity in the observable planar Fourier coefficients. 
    (b) The first three projected moments provide the quantities used for reconstruction. 
    (c) The forward frequency march recovers the isolating outer-shell coefficients in increasing harmonic degree. 
    (d) The reverse backfill recovers the remaining inner-shell coefficients in decreasing harmonic degree.}
    \label{fig:1}
\end{figure*}

\paragraph{Related work.}
In cryo-EM, constructive recovery from the first three projected moments was previously established in~\cite{liu2019rigorous}, under a full-rank condition on concatenated spherical-harmonic coefficient matrices of matching parity. In contrast, our approach relies on an explicit shell-separation condition on the signal representation, enabling a simpler reconstruction based on selected third-order moments and degree-wise linear systems.

Other moment-based approaches have established identifiability from second-order moments by exploiting additional structural assumptions or observational information. In~\cite{bendory2023autocorrelation}, the second projected moment was shown to identify molecular structures modeled as sparse sums of Gaussians, yielding a high-noise sample-complexity scale of $\sigma^4$. In~\cite{kileel2026two}, generic identifiability was established using second-order moments from two datasets with distinct viewing distributions, one uniform and the other nonuniform and unknown. Separately,~\cite{hoskins2026subspace} introduced a subspace method of moments for efficient joint reconstruction of the volume and viewing distribution, without establishing a corresponding global identifiability guarantee. These results are complementary to our constructive recovery guarantee from the first three projected moments of a single uniformly oriented dataset.

\section{Preliminaries}
\label{sec:prelim}

\paragraph{Notation.}
We write $\mathbb{R}^n$ and $\mathbb{C}$ for the $n$-dimensional real Euclidean space and the complex numbers, respectively. The orthogonal and rotation groups of $n\times n$ real matrices are denoted by $\mathrm{O}(n)$ and $\mathrm{SO}(n)$, respectively, and $\mathrm{Haar}(\mathrm{SO}(n))$ denotes the normalized Haar measure on $\mathrm{SO}(n)$. For $x,y\in\mathbb{R}^n$, we write $\langle x,y\rangle$ for the standard Euclidean inner product. Throughout, $\mathcal{F}_n$ denotes the Fourier transform on $\mathbb{R}^n$; explicitly, for $f\in L^1(\mathbb{R}^n)$,
\begin{align}
    \widehat{f}(\omega) = (\mathcal{F}_n f)(\omega) := \int_{\mathbb{R}^n} f(x)e^{-i\langle \omega,x\rangle}\,dx,
    \qquad \omega\in\mathbb{R}^n.
\end{align}
Expectations are understood whenever the relevant integrability conditions hold.

\paragraph{Noiseless projection moments.}
For the identifiability analysis, we work with the moments of the noiseless projection $P(R\cdot f)$, with $R\sim\mathrm{Haar}(\mathrm{SO}(3))$, rather than directly with the noisy observations~\eqref{eq:cryoEM}. At the population level, corresponding to $n\to\infty$, these signal-only moments can be isolated from the moments of the noisy observations by subtracting the explicit lower-order contributions of the additive Gaussian noise; see, e.g.,~\cite[Appendix~C.1]{balanov2026generalized}. We therefore formulate the theoretical results directly in terms of the signal-only projected moments and return to their finite-sample estimation from noisy observations in Section~\ref{sec:constructive-recovery}. Here $(R\cdot f)(x):=f(R^{-1}x)$, and, writing $x=(u,z)\in\mathbb{R}^2\times\mathbb{R}$, the tomographic projection is
\begin{align}
    P(f)(u):=\int_{\mathbb{R}}f(u,z)\,dz,
    \qquad u\in\mathbb{R}^2.
\end{align}

\paragraph{Handedness invariance.}
As is standard in single-particle cryo-EM, the volume cannot in general be distinguished from its mirror image, giving rise to the usual \emph{handedness} (or chirality) ambiguity~\cite{singer2011three,bendory2020single}. Thus, although the viewing rotations belong to $\mathrm{SO}(3)$, the distribution of the noiseless projections is invariant under the full orthogonal group $\mathrm{O}(3)$, and the volume is identifiable at best up to its global $\mathrm{O}(3)$ orbit.

\begin{lem}[Handedness invariance]
\label{lem:reflection-invariance}
For every real volume $f\in V$ and every $g\in\mathrm{O}(3)$,
\begin{align}
    P(R\cdot(g\cdot f)) \stackrel{d}{=} P(R\cdot f),
    \qquad
    R\sim\mathrm{Haar}(\mathrm{SO}(3)).
\end{align}
\end{lem}

\subsection{Fourier slice theorem and spherical-shell model}
\label{subsec:spherical-harmonics}

\paragraph{Fourier slice theorem.} The Fourier slice theorem makes the Fourier domain particularly convenient for analyzing tomographic projections. If $S$ denotes restriction to the equatorial plane, and $\widehat{f} = \mathcal{F}_3 f$, then by the Fourier slice theorem~\cite{bendory2020single,balanov2026group},
\begin{align}
    \mathcal{F}_2 P(R\cdot f)=S(R\cdot \widehat{f}).
    \label{eq:fourier-slice-heuristic}
\end{align}
In spherical coordinates $(r,\theta,\varphi)$, this restriction is simply $(S\widehat{f})(r,\varphi)=\widehat{f}(r,\pi/2,\varphi)$. Thus, tomographic projection becomes an equatorial restriction of the rotated Fourier transform.

\paragraph{Spherical-harmonic representation.} Motivated by this formulation, we represent $\widehat{f}$ in a finite-dimensional spherical-harmonic basis,
\begin{align}
    \widehat{f}(r,\theta,\varphi) = \sum_{\ell=0}^{L}\sum_{s=1}^{S_\ell}\sum_{m=-\ell}^{\ell} a_{\ell,m,s}\,\Psi_{\ell,m,s}(r,\theta,\varphi),
    \label{eq:continuous_model}
\end{align}
where $L$ is the harmonic bandlimit, $S_\ell$ is the number of radial functions associated with degree $\ell$, and $m=-\ell,\ldots,\ell$ indexes the corresponding spherical-harmonic components. We assume that the basis functions separate as
\begin{align}
    \Psi_{\ell,m,s}(r,\theta,\varphi) = F_{\ell,s}(r)Y_\ell^m(\theta,\varphi),
    \label{eq:spherical-harmonic-basis}
\end{align}
where $\{Y_\ell^m\}$ are the spherical harmonics and $\{F_{\ell,s}\}$ are radial functions. Such separable representations include, for example, spherical-prolate bases used in cryo-EM~\cite{sharon2020method,kileel2025fast}. We take the functions $\Psi_{\ell,m,s}$ to be orthonormal in $L^2(\mathbb{R}^3)$. Since the spherical harmonics are orthonormal on $\mathbb{S}^2$, this implies that, for each fixed $\ell$,
\begin{align}
    \int_0^\infty F_{\ell,s}(r)\overline{F_{\ell,t}(r)}\,r^2\,dr = \delta_{st}.
\end{align}
We refer to the collection of coefficients $a_{\ell,m,s}$ for fixed $\ell$
as the harmonic block of degree $\ell$.
\paragraph{Spherical-shell model.} We now specialize~\eqref{eq:continuous_model} to the finite-dimensional spherical-shell model used in our main result. The radial coordinate is discretized into shells; on the shell grid we take $F_{\ell,s}(r)=\delta_{sr}$ for $1\leq s\leq S_\ell$ (see Figure~\ref{fig:1}(a)). We also pass from the complex spherical harmonics to a basis adapted to the reality condition of the underlying volume. Specifically,
\begin{align}
    \widehat{f}(r,\theta,\varphi) = \sum_{\ell=0}^{L}\sum_{s=1}^{S_\ell}\sum_{m=-\ell}^{\ell} A_{\ell,m,s}\,\delta_{sr}\,H_{\ell,m}(\theta,\varphi),
    \label{eq.discrete_model}
\end{align}
where, following~\cite{bandeira2023estimation},
\begin{align}
    H_{\ell,m}(\theta,\varphi)=
    \begin{cases}
        \dfrac{1}{\sqrt{2}}\left(Y_\ell^{m}(\theta,\varphi)+(-1)^{\ell+m}Y_\ell^{-m}(\theta,\varphi)\right), & m>0,\\[1.2ex]
        i^\ell Y_\ell^{0}(\theta,\varphi), & m=0,\\[1.2ex]
        \dfrac{i}{\sqrt{2}}\left(Y_\ell^{|m|}(\theta,\varphi)-(-1)^{\ell+m}Y_\ell^{-|m|}(\theta,\varphi)\right), & m<0.
    \end{cases}
    \label{eq:real-spherical-harmonic-basis}
\end{align}
With this choice, a real-valued volume is represented by real coefficients $A_{\ell,m,s}$. For each degree $\ell$ and shell $s$, define
\begin{align}
    A_\ell[s] := \bigl(A_{\ell,-\ell,s},\ldots,A_{\ell,\ell,s}\bigr)^\top \in\mathbb{R}^{2\ell+1},
    \qquad 1\leq s\leq S_\ell.
    \label{eq:Al-shell-vector}
\end{align}
The matrix $A_\ell\in\mathbb{R}^{(2\ell+1)\times S_\ell}$ whose $s$th column is $A_\ell[s]$ encodes the harmonic block of degree $\ell$.

For the harmonic calculations below, it is convenient to retain the complex spherical-harmonic coordinates from~\eqref{eq:continuous_model}. Let
\begin{align}
    a_\ell[s] := \bigl(a_{\ell,-\ell,s},\ldots,a_{\ell,\ell,s}\bigr)^\top =U_\ell A_\ell[s],
    \label{eq:complex-real-change-of-basis}
\end{align}
where $U_\ell$ is the fixed unitary change-of-basis matrix between the orthonormal bases 
$\{Y_\ell^m\}$ and 
$\{H_{\ell,m}\}$ of the space of spherical harmonic polynomials of degree $\ell$.
Thus $a_\ell[s]$ and $A_\ell[s]$ contain exactly the same information. We use $a_{\ell,m,s}$ in Wigner-$D$ and Clebsch--Gordan formulas and $A_\ell[s]$ for the real coefficient blocks recovered by the algorithm.

\subsection{Fourier representation of a projected observation}
\label{subsec:projection_mixing}

The effect of tomographic projection on the spherical-harmonic representation is explicit. We use the standard complex spherical harmonics
\begin{align}
    Y_\ell^m(\theta,\varphi) = N_{\ell m}P_\ell^m(\cos\theta)e^{im\varphi},
    \qquad
    N_{\ell m} := \sqrt{\frac{2\ell+1}{4\pi}\frac{(\ell-m)!}{(\ell+m)!}},
\end{align}
where $P_\ell^m$ denotes the associated Legendre function with the Condon-Shortley phase convention. By~\eqref{eq:spherical-harmonic-basis}, restriction to the equatorial plane gives
\begin{align}
    S\Psi_{\ell,m,s}(r,\varphi) = F_{\ell,s}(r)N_{\ell m}P_\ell^m(0)e^{im\varphi}.
    \label{eq:spherical-harmonic-restriction}
\end{align}
Since $P_\ell^m(0)=0$ whenever $\ell+m$ is odd, a degree-$\ell$ component survives at planar frequency $m$ only when $\ell$ and $m$ have the same parity.

The viewing rotation acts independently within each harmonic degree through the Wigner $D$-matrices. Specifically, let $D^\ell(R)\in\mathbb{C}^{(2\ell+1)\times(2\ell+1)}$ denote the degree-$\ell$ irreducible representation of $\mathrm{SO}(3)$, with entries defined by
\begin{align}
    Y_\ell^m(R^{-1}\omega) = \sum_{k=-\ell}^{\ell}D_{km}^{\ell}(R)Y_\ell^k(\omega),
    \qquad \omega\in\mathbb{S}^2.
    \label{eq:wigner-action}
\end{align}
Substituting~\eqref{eq:wigner-action} into~\eqref{eq:continuous_model} gives
\begin{align}
    (R\cdot\widehat{f})(r,\theta,\varphi) = \sum_{\ell=0}^{L}\sum_{s=1}^{S_\ell}\sum_{k=-\ell}^{\ell} \left(\sum_{m=-\ell}^{\ell}D_{km}^{\ell}(R)a_{\ell,m,s}\right) F_{\ell,s}(r)Y_\ell^k(\theta,\varphi).
    \label{eq:rotated-fourier-expansion}
\end{align}
Thus, the latent rotation mixes the three-dimensional indices $m$ within each fixed degree $\ell$, without coupling different harmonic degrees.

\paragraph{Two-dimensional Fourier transform of a projection.}
Combining~\eqref{eq:rotated-fourier-expansion} with the Fourier slice theorem~\eqref{eq:fourier-slice-heuristic} and the equatorial restriction~\eqref{eq:spherical-harmonic-restriction}, the two-dimensional Fourier transform of the corresponding noiseless projection is
\begin{align}
    \widehat{X}_R(r,\varphi) &:= \mathcal{F}_2P(R\cdot f)(r,\varphi)
    \nonumber\\
    &=
    \sum_{\ell=0}^{L} \sum_{s=1}^{S_\ell} \sum_{k=-\ell}^{\ell} \left( \sum_{m=-\ell}^{\ell} D_{km}^{\ell}(R)a_{\ell,m,s} \right) F_{\ell,s}(r) N_{\ell k}P_\ell^k(0)e^{ik\varphi}.
    \label{eq:projected-observation-fourier}
\end{align}
Here, $m$ indexes the three-dimensional spherical-harmonic components within a fixed degree $\ell$, while $k$ is the planar Fourier frequency of the projected image. Since $P_\ell^k(0)=0$ whenever $\ell+k$ is odd, a planar frequency $k$ can receive contributions only from degrees $\ell\geq|k|$ having the same parity as $k$.

In the spherical-shell model~\eqref{eq.discrete_model}, let $\rho_s$ denote the radius of shell $s$, and let $\widehat{X}_{R,k}[s]$ denote the $k$th planar Fourier coefficient of the noiseless projected signal on that shell. Thus,
\begin{align}
    \widehat{X}_R(\rho_s,\varphi) = \sum_{k\in\mathbb Z} \widehat{X}_{R,k}[s]e^{ik\varphi} \qquad \widehat{X}_{R,k}[s] = \frac{1}{2\pi} \int_0^{2\pi} \widehat{X}_R(\rho_s,\varphi)e^{-ik\varphi}\,d\varphi.
    \label{eq:projected-angular-expansion}
\end{align}
Using~\eqref{eq:projected-observation-fourier}, these coefficients satisfy
\begin{align}
    \widehat{X}_{R,k}[s] = \sum_{\substack{ \ell\geq|k|,\;\ell\equiv k\!\!\!\pmod 2\\ s\leq S_\ell }} N_{\ell k}P_\ell^k(0) \sum_{m=-\ell}^{\ell}
    D_{km}^{\ell}(R)a_{\ell,m,s}.
    \label{eq:projected-shell-coefficient}
\end{align}
Hence, a fixed observable pair $(k,s)$ can receive contributions from several three-dimensional harmonic degrees,
$\ell=|k|,|k|+2,|k|+4,\ldots$, subject to $s\leq S_\ell$ (see Figure~\ref{fig:1}(a)). This harmonic degree mixing is the key effect of tomographic projection used in the recovery argument below. For the noisy observations~\eqref{eq:cryoEM}, the corresponding coefficients are
\begin{align}
    \widehat{y}_{j,k}[s] = \widehat{X}_{R_j,k}[s] + \widehat\varepsilon_{j,k}[s].
    \label{eq:noisy-projected-shell-coefficient}
\end{align}
These noisy coefficients are used to estimate the signal-only projected moments from finitely many observations.

\subsection{Unprojected and projected moment tensors}
\label{subsec:projected-moments}

We now describe the population moments associated with the Fourier-domain representation introduced above and compare the information retained before and after tomographic projection. 

For each $d\geq1$, and $\eta_1,\ldots,\eta_d\in\mathbb{R}^3$, we define the $d$th-order unprojected Fourier moment tensor by
\begin{align}
    M_d(f)(\eta_1,\ldots,\eta_d) := \mathbb{E}_{R\sim\mathrm{Haar}(\mathrm{SO}(3))} \left[ \prod_{j=1}^{d}(R\cdot\widehat{f})(\eta_j) \right].
    \label{eq:Md-Fourier-definition}
\end{align}
Likewise, recalling from~\eqref{eq:projected-observation-fourier} that $\widehat{X}_R=S(R\cdot\widehat{f})$ is the two-dimensional Fourier transform of a noiseless projection, we define for $\tau_j=(r_j,\varphi_j)\in\mathbb{R}^2$,
\begin{align}
    PM_d(f)(\tau_1,\ldots,\tau_d) := \mathbb{E}_{R\sim\mathrm{Haar}(\mathrm{SO}(3))} \left[ \prod_{j=1}^{d}\widehat{X}_R(\tau_j) \right].
    \label{eq:projected-moment-definition}
\end{align}
By the Fourier slice theorem~\eqref{eq:fourier-slice-heuristic}, equivalently $PM_d=S^{\otimes d}M_d$. Thus, the difference between the unprojected and projected moments is entirely due to the equatorial restriction.

In the spherical-shell model, we represent the projected moments through the planar Fourier coefficients~\eqref{eq:projected-shell-coefficient}. For $d\geq1$, define
\begin{align}
    \mathcal{P}_d(k_1,s_1,\ldots,k_d,s_d) := \mathbb{E}_{R\sim\mathrm{Haar}(\mathrm{SO}(3))} \left[ \prod_{j=1}^{d}\widehat{X}_{R,k_j}[s_j] \right].
    \label{eq:projected-shell-moment}
\end{align}
These are the observable projected-moment coefficients used throughout the recovery argument and estimated empirically from finitely many noisy observations in Section~\ref{sec:constructive-recovery}.

\paragraph{Haar averages of Wigner $D$-matrices.}
The first three moment formulas follow from the corresponding Haar averages of the Wigner $D$-matrices introduced in~\eqref{eq:wigner-action}. With normalized Haar measure on $\mathrm{SO}(3)$,
\begin{align}
    \mathbb{E}_{R\sim\mathrm{Haar}(\mathrm{SO}(3))}\!\left[D_{km}^{\ell}(R)\right] =\delta_{\ell0}\delta_{k0}\delta_{m0}.
    \label{eq:wigner-first-moment}
\end{align}
At second order, Schur orthogonality gives
\begin{align}
    \mathbb{E}_{R\sim\mathrm{Haar}(\mathrm{SO}(3))}\!\left[
        D_{k_1m_1}^{\ell_1}(R)
        D_{k_2m_2}^{\ell_2}(R)
    \right]
    =
    \frac{(-1)^{m_1+k_1}}{2\ell_1+1}
    \delta_{\ell_1\ell_2}
    \delta_{k_1,-k_2}
    \delta_{m_1,-m_2}.
    \label{eq:wigner-second-moment}
\end{align}
At third order,
\begin{align}
    \mathbb{E}_{R\sim\mathrm{Haar}(\mathrm{SO}(3))}\!\left[
        D_{k_1m_1}^{\ell_1}(R)
        D_{k_2m_2}^{\ell_2}(R)
        D_{k_3m_3}^{\ell_3}(R)
    \right]
    =
    \begin{pmatrix}
        \ell_1 & \ell_2 & \ell_3\\
        k_1 & k_2 & k_3
    \end{pmatrix}
    \begin{pmatrix}
        \ell_1 & \ell_2 & \ell_3\\
        m_1 & m_2 & m_3
    \end{pmatrix},
    \label{eq:wigner-third-moment}
\end{align}
where the parentheses denote Wigner $3j$ symbols. In particular, the third-order average vanishes unless $k_1+k_2+k_3=0$, $m_1+m_2+m_3=0$, and
\begin{align}
    |\ell_2-\ell_3|\leq\ell_1\leq\ell_2+\ell_3.
    \label{eq:triangle-condition}
\end{align}
We denote the latter condition by $\Delta(\ell_1,\ell_2,\ell_3)$.

\paragraph{First and second moments.}
Applying~\eqref{eq:wigner-first-moment} to~\eqref{eq:rotated-fourier-expansion}, only the trivial representation $\ell=0$ survives. Hence
\begin{align}
    M_1(f)(\eta) = \sum_{s=1}^{S_0}a_{0,0,s}\Psi_{0,0,s}(\eta).
    \label{eq:M1-unprojected}
\end{align}
Since the degree-zero basis functions are radial, the projected first moment retains the same coefficients. Thus, both $M_1$ and $PM_1$ determine the degree-zero block $A_0$.

At second order,~\eqref{eq:wigner-second-moment} implies that only equal harmonic degrees couple. The unprojected second moment therefore determines the degree-wise quadratic invariants. Under uniformly distributed viewing rotations, the projected second moment determines the same degree-wise Gram matrices~\cite{balanov2026group}. In the real coordinates~\eqref{eq:real-spherical-harmonic-basis}, these invariants are the Gram matrices
\begin{align}
    G_\ell:=A_\ell^\top A_\ell,
    \qquad \ell=0,\ldots,L.
    \label{eq:degree-wise-gram}
\end{align}
Consequently, $PM_2$ determines each real harmonic block $A_\ell$ only up to an independent left orthogonal transformation.

\paragraph{Third moment.}
At third order,~\eqref{eq:wigner-third-moment} couples different harmonic
degrees through the Clebsch-Gordan decomposition~\cite{bendory2025orbit}.
For an admissible degree triple $(\ell_1,\ell_2,\ell_3)$, define the cubic
invariant
\begin{align}
    I_3(\ell_1,s_1,\ell_2,s_2,\ell_3,s_3) &:= \sum_{\substack{ m_1+m_2+m_3=0\\ |m_i|\leq\ell_i }} (-1)^{m_1} \left\langle \ell_2m_2\,\ell_3m_3 \,\middle|\, \ell_1(-m_1) \right\rangle
    \nonumber\\
    &\qquad\times
    a_{\ell_1,m_1,s_1} a_{\ell_2,m_2,s_2} a_{\ell_3,m_3,s_3}.
    \label{eq:I3}
\end{align}
Here $\langle \ell_2m_2\,\ell_3m_3\mid\ell_1(-m_1)\rangle$ is the Clebsch--Gordan coefficient coupling $(\ell_2,m_2)$ and $(\ell_3,m_3)$ to $(\ell_1,-m_1)$.

The relation between these invariants and the unprojected third moment is
most transparent by introducing the known invariant tensor
\begin{align}
    \mathcal{B}^{(3)}_{\ell_1,s_1,\ell_2,s_2,\ell_3,s_3} (\eta_1,\eta_2,\eta_3) & :=
    \frac{1}{2\ell_1+1} \sum_{\substack{ q_1+q_2+q_3=0\\ |q_i|\leq\ell_i }} (-1)^{q_1} \left\langle   \ell_2q_2\,\ell_3q_3 \,\middle|\, \ell_1(-q_1) \right\rangle
    \nonumber\\
    &\qquad\qquad\times
    \Psi_{\ell_1,q_1,s_1}(\eta_1) \Psi_{\ell_2,q_2,s_2}(\eta_2) \Psi_{\ell_3,q_3,s_3}(\eta_3).
    \label{eq:third-invariant-tensor}
\end{align}
Then,~\eqref{eq:wigner-third-moment} yields the decomposition
\begin{align}
    M_3(f)(\eta_1,\eta_2,\eta_3) = \sum_{\substack{ \ell_1,\ell_2,\ell_3\\ s_1,s_2,s_3\\ \Delta(\ell_1,\ell_2,\ell_3) }} I_3(\ell_1,s_1,\ell_2,s_2,\ell_3,s_3)\, \mathcal{B}^{(3)}_{\ell_1,s_1,\ell_2,s_2,\ell_3,s_3} (\eta_1,\eta_2,\eta_3).
    \label{eq:M3-invariant-expansion}
\end{align}
For each admissible degree and shell triple, $\mathcal{B}^{(3)}_{\ell_1,s_1,\ell_2,s_2,\ell_3,s_3}$ spans the one-dimensional $\mathrm{SO}(3)$-invariant subspace of the 
$\mathrm{SO}(3)$ representation spanned by the products $$\{\Psi_{\ell_1, m_1, s_1} (\eta_1) \Psi_{\ell_2, m_2, s_2}(\eta_2) \Psi_{\ell_3, m_3, s_3}(\eta_3)\}_{-\ell_i \leq m_i \leq \ell_i} , $$
corresponding tensor-product block. Moreover, distinct degree and shell blocks are orthogonal. Thus, the coefficients of $M_3$ in this invariant decomposition are precisely the cubic invariants $I_3$, and the unprojected third moment determines them individually.

Tomographic projection destroys this blockwise separation. Combining the shell coefficient formula~\eqref{eq:projected-shell-coefficient} with the third-order identity~\eqref{eq:wigner-third-moment} gives
\begin{align}
    \mathcal{P}_3(k_1,s_1,k_2,s_2,k_3,s_3) = \sum_{\substack{ \ell_i\geq|k_i|,\;\ell_i\equiv k_i\!\!\!\pmod 2\\ s_i\leq S_{\ell_i},\;\Delta(\ell_1,\ell_2,\ell_3) }} C_{\ell_1,k_1,\ell_2,k_2,\ell_3,k_3}\, I_3(\ell_1,s_1,\ell_2,s_2,\ell_3,s_3),
    \label{eq:P3}
\end{align}
where
\begin{align}
    C_{\ell_1,k_1,\ell_2,k_2,\ell_3,k_3} := \frac{(-1)^{k_1}}{2\ell_1+1} \left\langle \ell_2k_2\,\ell_3k_3 \,\middle|\, \ell_1(-k_1) \right\rangle \prod_{i=1}^{3}N_{\ell_i k_i}P_{\ell_i}^{k_i}(0).
    \label{eq:projected-third-moment-coefficient}
\end{align}

For a fixed observable frequency triple $(k_1,k_2,k_3)$, the possible contributing degrees are therefore $\ell_i=|k_i|,|k_i|+2,|k_i|+4,\ldots$, subject also to~\eqref{eq:triangle-condition} and the shell constraints. This is the essential difference between the unprojected and projected third moments: $M_3$ separates the invariants $I_3$ into distinct invariant blocks, whereas an observable coefficient $\mathcal{P}_3$ generally mixes invariants associated with several harmonic degree triples. The main result below shows that the spherical-shell structure allows these contributions to be separated and the volume to be recovered up to its global $\mathrm{O}(3)$ ambiguity.

\section{Main result}
\label{sec:main-result}

We now state our main identifiability result for the spherical-shell model~\eqref{eq.discrete_model}. Section~\ref{subsec:projected-moments} shows that the first projected moment determines the degree-zero block, whereas the second projected moment determines each harmonic block only up to an independent orthogonal transformation. The role of the third projected moment is therefore to couple these blocks and resolve the remaining ambiguities. The main difficulty is that the cubic invariants $I_3$ are not observed individually: by~\eqref{eq:P3}, a fixed planar-frequency triple can receive contributions from several harmonic degree triples.

Our recovery argument exploits the radial-shell structure to separate these mixed contributions. We adopt the convention $S_j:=0$, for $j>L$, so that the same condition applies to the highest harmonic degrees.

\begin{assum}[Decay-by-three shell counts]
\label{assum:decay}
For every $0\leq\ell\leq L$,
\begin{align}
    S_\ell-S_{\ell+2}\geq3.
    \label{eq:decay}
\end{align}
\end{assum}
In particular, for $s$ in the interval $(S_{\ell+2}, S_\ell]$
the coefficients $A_{\ell', s,m'}$ are zero for $\ell' = \ell +2, \ell +4, \ldots$. We refer to the shells indexed by $s \in (S_{\ell+2}, S_\ell]$ as the {\em isolating shells} of degree $\ell$ and denote this set by $T_{\ell}^{\mathrm{iso}}$. The requirement that there are at least three isolating shells in each degree provides the multiplicity needed for the linear systems in the frequency-marching procedure to be solvable. The convention that $S_{j}=0$ for $j > L$ implies that $S_{L-1}$ and $S_L$ are both at least three.


The proof additionally uses the ability to extract shell coefficients of the projected third moment separately. This is automatic in the spherical-shell model: fixing $(s_1,s_2,s_3)$ in~\eqref{eq:projected-shell-moment} selects the corresponding radial coefficient. For more general radial bases, an analogous argument requires sufficient linear independence across harmonic degrees. We return to this modeling issue in Section~\ref{sec:discussion}.

\begin{thm}
\label{thm.shells}
Let $f$ be a generic real volume whose Fourier transform belongs to the spherical-shell model~\eqref{eq.discrete_model}, and suppose that Assumption~\ref{assum:decay} holds. Then the first three projected moments $PM_1(f)$, $PM_2(f)$, and $PM_3(f)$ determine the global $\mathrm{O}(3)$ orbit of $f$.
\end{thm}

We next describe the constructive recovery procedure underlying Theorem~\ref{thm.shells}. We first present the reconstruction mechanism, consisting of a forward frequency march followed by a reverse backfill. We then prove that, under Assumption~\ref{assum:decay}, the required linear systems are generically solvable at every step, thereby establishing Theorem~\ref{thm.shells}.

\subsection{Constructive recovery}
Theorem~\ref{thm.shells} is proved by an explicit degree-by-degree recovery procedure. The first projected moment determines $A_0$, while the second projected moment determines the Gram matrices $G_\ell=A_\ell^\top A_\ell$. In particular, factoring $G_1$ determines $A_1$ up to a left orthogonal transformation. We fix one such factor, thereby choosing a representative of the global $\mathrm{O}(3)$ orbit.

The remaining coefficients are recovered from the observable third-order coefficients $\mathcal{P}_3$ (as derived in~\eqref{eq:P3}) in two passes. The forward frequency march of Section~\ref{sec:outer-shell}, formalized in Proposition~\ref{prop:outer-shell}, uses isolating shells and outer angular bands to recover the cubic invariants $I_3$ and hence the coefficients $A_\ell[s]$ for $s\in T_\ell^{\mathrm{iso}}$ in increasing harmonic degree. The reverse backfill of Section~\ref{sec:backfill}, formalized in Proposition~\ref{prop:backfill}, proceeds in decreasing harmonic degree: after subtracting the already known higher-degree contributions from $\mathcal{P}_3$, the remaining cubic invariants determine the inner-shell coefficients $A_\ell[s]$ for $s\leq S_{\ell+2}$. Algorithm~\ref{alg:population-recovery} summarizes the resulting population reconstruction, while Figure~\ref{fig:1}(b)-(d) illustrates the forward and reverse stages of the procedure.

The key ingredient in the forward pass is that outer angular bands together with isolating shells reduce an observable projected-moment coefficient to a single cubic invariant.

\begin{lem}[Highest-weight reduction]
\label{lem:highest-weight}
Recall the definition of $C_{\ell_1,k_1,\ell_2,k_2,\ell_3,k_3}$ from~\eqref{eq:projected-third-moment-coefficient}. Let $k_1,k_2,k_3\in\mathbb Z$ satisfy $k_1+k_2+k_3=0$, and set $\ell_i:=|k_i|\leq L$. If $s_i\in T_{\ell_i}^{\mathrm{iso}}$ for $i=1,2,3$, then
\begin{align}
    \mathcal{P}_3(k_1,s_1,k_2,s_2,k_3,s_3) = C_{\ell_1,k_1,\ell_2,k_2,\ell_3,k_3}\, I_3(\ell_1,s_1,\ell_2,s_2,\ell_3,s_3).
    \label{eq:hw-reduction}
\end{align}
Moreover, $C_{\ell_1,k_1,\ell_2,k_2,\ell_3,k_3}\neq 0$.
\end{lem}

\begin{proof}
Consider a nonzero term in~\eqref{eq:P3} with harmonic degrees $(\ell_1',\ell_2',\ell_3')$. Since $P_{\ell_i'}^{k_i}(0)\neq 0$ only when $\ell_i'+k_i$ is even, and $\ell_i=|k_i|$ has the same parity as $k_i$, any contributing degree must satisfy $\ell_i' \in \{\ell_i,\ell_i+2,\ell_i+4,\ldots\}$. On the other hand, $s_i\in T_{\ell_i}^{\mathrm{iso}}$ implies $s_i>S_{\ell_i+2}$. Assumption~\ref{assum:decay} makes the shell counts decrease along each parity sequence, so $s_i>S_{\ell_i'}$ for every $\ell_i'\geq\ell_i+2$. The constraint $s_i\leq S_{\ell_i'}$ in~\eqref{eq:P3} therefore forces $\ell_i'=\ell_i$ for each $i$, and the sum reduces to~\eqref{eq:hw-reduction}.

It remains to verify that the coefficient is nonzero. Since $\ell_i=|k_i|$, we have $P_{\ell_i}^{k_i}(0)\neq 0$. Moreover, $k_1+k_2+k_3=0$ implies the triangle inequalities for $(\ell_1,\ell_2,\ell_3)$, and the corresponding extremal Clebsch--Gordan coefficient $\left\langle \ell_2k_2\,\ell_3k_3  \,\middle|\,  \ell_1(-k_1) \right \rangle$ is nonzero. Hence~\eqref{eq:projected-third-moment-coefficient} gives $C_{\ell_1,k_1,\ell_2,k_2,\ell_3,k_3}\neq 0$.
\end{proof}

\begin{algorithm}[H]
    \caption{Population recovery from the first three projected moments}
    \label{alg:population-recovery}
    \begin{algorithmic}[1]
        \Statex \textbf{Input:} $PM_1(f)$, $PM_2(f)$, and $PM_3(f)$; bandlimit $L$; shell counts $S_0,\ldots,S_L$ satisfying Assumption~\ref{assum:decay}.
        \Statex \textbf{Output:} Coefficients $\{A_\ell[s]\}$ up to a global $\mathrm{O}(3)$ transformation.

        \State Recover $A_0$ from $PM_1(f)$.
        \State Recover $G_1=A_1^\top A_1$ from $PM_2(f)$, factor $G_1$, and fix a representative $A_1$ of the global $\mathrm{O}(3)$ orbit.

        \Statex \underline{Forward frequency march} (Section~\ref{sec:outer-shell}).
        \For{$\ell=2,\ldots,L$}
            \For{$s\in T_\ell^{\mathrm{iso}}$}
                \State Recover the cubic invariants $I_3$ involving $A_\ell[s]$ from $\mathcal{P}_3$ using~\eqref{eq:outer-reduced}.
                \State Recover $A_\ell[s]$ by solving the linear system~\eqref{eq:outer-linear-equation} over the admissible lower-degree pairs.
            \EndFor
        \EndFor

        \Statex \underline{Reverse backfill} (Section~\ref{sec:backfill}).
        \For{$\ell=L-2,\ldots,2$}
            \For{$s=1,\ldots,S_{\ell+2}$}
                \State Recover the degree-$\ell$ cubic invariants $I_3$ by subtracting the known contributions of degrees $\ell+2,\ell+4,\ldots$ from $\mathcal{P}_3$ using~\eqref{eq:backfill-decomposition}.
                \State Recover $A_\ell[s]$ by solving the linear system~\eqref{eq:backfill-linear-equation}.
            \EndFor
        \EndFor

        \State \Return $\{A_\ell[s]\}$.
    \end{algorithmic}
\end{algorithm}

\subsection{Outer-shell recovery}
\label{sec:outer-shell}

We first recover the coefficients on the isolating shells by marching upward in harmonic degree.

\begin{proposition}[Outer-shell recovery]
\label{prop:outer-shell}
Suppose that $A_1$ has been fixed up to the global $\mathrm{O}(3)$ ambiguity. Then, for a generic signal, the projected third moment determines $A_\ell[s]$, for $2\leq\ell\leq L$, and $s\in T_\ell^{\mathrm{iso}}$.
\end{proposition}

\begin{proof}
We proceed by induction on $\ell$. Suppose all isolating-shell vectors below degree $\ell$ have been recovered. Through~\eqref{eq:complex-real-change-of-basis}, their complex-coordinate vectors are therefore also known. Fix $s\in T_\ell^{\mathrm{iso}}$ and choose planar frequencies $k_1,k_2,k$ satisfying $k_1+k_2+k=0$, $|k|=\ell$, and $\ell_i:=|k_i|<\ell$ for $i=1,2$. For any $s_i\in T_{\ell_i}^{\mathrm{iso}}$, Lemma~\ref{lem:highest-weight} gives
\begin{align}
    \mathcal{P}_3(k_1,s_1,k_2,s_2,k,s) =C_{\ell_1,k_1,\ell_2,k_2,\ell,k}\, I_3(\ell_1,s_1,\ell_2,s_2,\ell,s),
    \label{eq:outer-reduced}
\end{align}
where the prefactor is known and nonzero.

Since the first two coefficient vectors are already known,~\eqref{eq:I3} is linear in the unknown complex vector $a_\ell[s]$:
\begin{align}
    I_3(\ell_1,s_1,\ell_2,s_2,\ell,s) =\sum_{m=-\ell}^{\ell} B_m\!\left(a_{\ell_1}[s_1],a_{\ell_2}[s_2]\right)a_{\ell,m,s},
    \label{eq:outer-linear-equation}
\end{align}
where
\begin{align}
    B_m\!\left(a_{\ell_1}[s_1],a_{\ell_2}[s_2]\right) :=\sum_{\substack{m_1+m_2+m=0\\ |m_i|\leq\ell_i}} (-1)^{m_1} \left\langle \ell_2m_2\,\ell m\,\middle|\,\ell_1(-m_1)\right\rangle a_{\ell_1,m_1,s_1}a_{\ell_2,m_2,s_2}.
    \label{eq:outer-linear-coefficients}
\end{align}
Using $a_\ell[s]=U_\ell A_\ell[s]$, these equations form a real linear system for the $2\ell+1$ entries of $A_\ell[s]$. 

It remains to show that sufficiently many choices of $(k_1,k_2,s_1,s_2)$ yield a full-rank system. The coefficients in \eqref{eq:outer-linear-coefficients} are quadratic polynomials in the already recovered lower-degree coefficients and may be viewed as linear forms in the corresponding quadratic monomials. As in~\cite[Proposition~4]{bendory2025orbit}, the resulting coefficient matrix has rank $2\ell+1$ for a generic choice of the lower-degree vectors; the fixed invertible change of basis $U_\ell$ preserves this rank. Assumption~\ref{assum:decay} provides at least three isolating shells at each degree, supplying the shell multiplicity required by this construction. Thus $A_\ell[s]$ is uniquely determined for every $s\in T_\ell^{\mathrm{iso}}$, completing the induction.
\end{proof}

Thus, after the forward march, all vectors $A_\ell[s]$ with $s\in T_\ell^{\mathrm{iso}}$ are known. In particular, since $S_{L+2}=S_{L+1}=0$, all shells of degrees $L$ and $L-1$ are recovered.

\subsection{Inner-shell backfill}
\label{sec:backfill}

It remains to recover the coefficients on the inner shells $s\leq S_{\ell+2}$. We do so by descending harmonic degree. At each step, isolating shells fix the first two harmonic degrees in $\mathcal{P}_3$, while contributions from already recovered higher degrees in the third factor can be subtracted.

\begin{proposition}[Inner-shell backfill]
\label{prop:backfill}
Fix $2\leq\ell<L$ and suppose that $A_{\ell'}[s]$ is known for every $\ell'>\ell$ and every $1\leq s\leq S_{\ell'}$. Then, for a generic signal, the projected third moment determines $A_\ell[s]$ for every $s\leq S_{\ell+2}$.
\end{proposition}

\begin{proof}
Fix $s\leq S_{\ell+2}$ and choose planar frequencies $k_1,k_2,k$ satisfying $k_1+k_2+k=0$ and $|k|=\ell$. Set $\ell_i:=|k_i|$ and choose $s_i\in T_{\ell_i}^{\mathrm{iso}}$ for $i=1,2$. Since the first two shell indices are isolating, the first two harmonic degrees in~\eqref{eq:P3} are forced to be $\ell_1$ and $\ell_2$. The third harmonic degree can take the values $\ell,\ell+2,\ell+4,\ldots$ for which $s\leq S_{\ell'}$. Consequently,
\begin{align}
    \mathcal{P}_3(k_1,s_1,k_2,s_2,k,s) &=C_{\ell_1,k_1,\ell_2,k_2,\ell,k}\, I_3(\ell_1,s_1,\ell_2,s_2,\ell,s)
    \nonumber\\
    &\quad+ \sum_{\substack{\ell'>\ell,\;\ell'\equiv\ell\!\!\!\pmod 2\\ s\leq S_{\ell'}}} C_{\ell_1,k_1,\ell_2,k_2,\ell',k}\, I_3(\ell_1,s_1,\ell_2,s_2,\ell',s).
    \label{eq:backfill-decomposition}
\end{align}
Every term in the sum over $\ell'>\ell$ is known: the first two coefficient vectors were recovered in the forward march, while $A_{\ell'}[s]$ is known by the reverse-induction hypothesis. Since $C_{\ell_1,k_1,\ell_2,k_2,\ell,k}\neq 0$, subtracting the higher-degree terms determines $I_3(\ell_1,s_1,\ell_2,s_2,\ell,s)$.

As in~\eqref{eq:outer-linear-equation}, this invariant is linear in $a_\ell[s]$,
\begin{align}
    I_3(\ell_1,s_1,\ell_2,s_2,\ell,s) = \sum_{m=-\ell}^{\ell} B_m\!\left(a_{\ell_1}[s_1],a_{\ell_2}[s_2]\right)a_{\ell,m,s}.
    \label{eq:backfill-linear-equation}
\end{align}
After composing with $a_\ell[s]=U_\ell A_\ell[s]$, varying the admissible frequencies and isolating shells yields the same generic full-rank real system established in Proposition~\ref{prop:outer-shell}. Hence $A_\ell[s]$ is uniquely determined for a generic signal.
\end{proof}

Starting from the highest degrees and applying Proposition~\ref{prop:backfill} successively for decreasing $\ell$ recovers every remaining block with $\ell\geq2$. The blocks $A_0$ and $A_1$ were already fixed by the first two projected moments. Together with Proposition~\ref{prop:outer-shell}, this determines all coefficient matrices $A_\ell$ up to the single global $\mathrm{O}(3)$ ambiguity and completes the proof of Theorem~\ref{thm.shells}.

\section{Numerical experiments}
\label{sec:constructive-recovery}

The proof of Theorem~\ref{thm.shells} is constructive and directly yields a recovery algorithm from the first three projected moments. We first describe its finite-sample implementation, and then introduce a complementary direct optimization scheme that estimates the volume coefficients by fitting the projected moments jointly. We evaluate both approaches numerically on synthetic signals and molecular volumes. The experiments examine the effect of Assumption~\ref{assum:decay}, robustness to additive noise, the dependence of the reconstruction error on the number of observations, and the performance of direct moment-based optimization beyond the structural conditions required by the constructive proof.

\subsection{Finite-sample implementation}
\label{subsec:finite-sample-implementation}

We implement the population recovery procedure of Algorithm~\ref{alg:population-recovery} from finitely many noisy projections $\{y_j\}_{j=1}^n$. For the $j$th observation, let $\widehat{y}_{j,k}[s]$ denote its Fourier-domain coefficient at shell $s$ and planar frequency $k$, as defined in~\eqref{eq:noisy-projected-shell-coefficient}. For $d=1,2,3$, we first form the empirical moments of the noisy coefficients,
\begin{align}
    \widetilde{Q}_d(k_1,s_1,\ldots,k_d,s_d) := \frac{1}{n} \sum_{j=1}^{n} \prod_{i=1}^{d} \widehat{y}_{j,k_i}[s_i].
    \label{eq:empirical-noisy-projected-moments}
\end{align}
These estimate the corresponding population moments of the signal-plus-noise observations, rather than the signal-only coefficients $\mathcal{P}_d$ defined in~\eqref{eq:projected-shell-moment}.

Since the additive noise is Gaussian with known covariance, its contribution to moments of orders $d\leq3$ consists of explicit lower-order terms. We subtract these terms from $\widetilde{Q}_d$ using the standard Gaussian moment corrections; see, e.g.,~\cite[Appendix~C.1]{balanov2026generalized}. We denote the resulting debiased estimator by
\begin{align}
    \widehat{\mathcal{P}}_d = \widetilde{Q}_d - \text{noise correction},
    \qquad d=1,2,3,
    \label{eq:debiased-projected-moments}
\end{align}
where the correction is zero at first order and consists of the known noise-covariance contributions at second and third order. Thus, $\widehat{\mathcal{P}}_d$ estimates the signal-only projected moment $\mathcal{P}_d$, and these debiased moments are the quantities supplied to the finite-sample reconstruction algorithms below.

The finite-sample reconstruction follows Algorithm~\ref{alg:population-recovery}, with the population moment coefficients $\mathcal{P}_d$ replaced by their empirical estimates $\widehat{\mathcal{P}}_d$. The first two empirical moments yield estimates $\widehat{A}_0$ and $\widehat{G}_1$ of $A_0$ and $G_1=A_1^\top A_1$, respectively, from which a representative $\widehat{A}_1$ is chosen to fix the global $\mathrm{O}(3)$ ambiguity. The forward and reverse passes then recover estimates $\widehat{A}_\ell[s]$ of the remaining coefficient vectors. In each pass, the exact linear equations derived in Sections~\ref{sec:outer-shell} and~\ref{sec:backfill} are replaced by their empirical counterparts, constructed from $\widehat{\mathcal{P}}_3$ and the coefficient estimates recovered in the preceding steps, and solved by least squares. Whenever the Clebsch--Gordan formulas are evaluated, the estimated real coefficient vectors are converted to complex spherical-harmonic coordinates through $\widehat{a}_\ell[s]=U_\ell\widehat{A}_\ell[s]$, in accordance with~\eqref{eq:complex-real-change-of-basis}.

\subsection{Experimental setup}
\label{sec:numerical-setup}

We use two molecular volumes: the TRPV1 structure (EMD-8117) and the \emph{Plasmodium falciparum} 80S ribosome (EMD-2660), both available from the Electron Microscopy Data Bank (EMDB).\footnote{\url{https://www.ebi.ac.uk/emdb/}} We center each volume, restrict it to the unit ball, downsample it to $31^3$ voxels, and normalize it to unit Frobenius norm.

For each experiment, we sample the three-dimensional Fourier transform on $K$ concentric spherical shells with equally spaced radii
\begin{align}
    \rho_s = \left(s-\tfrac{1}{2}\right)\frac{\rho_{\max}}{K},
    \qquad
    s=1,\ldots,K.
    \label{eq:numerical-shell-radii}
\end{align}
On each shell, we expand the harmonic dependence in the real spherical-harmonic basis~\eqref{eq:real-spherical-harmonic-basis} up to degree $L$. Thus, the numerical representation is precisely the spherical-shell model~\eqref{eq.discrete_model}, with the shell index $s$ corresponding to the common physical radius $\rho_s$. The coefficients $A_{\ell,m,s}$ are obtained independently on each shell by spherical-harmonic projection.

For each observation, we draw $R_j$ independently from Haar measure and generate $y_j$ according to the cryo-EM model~\eqref{eq:cryoEM}. After applying the two-dimensional Fourier transform, we use the planar frequency-shell coefficients $\widehat{y}_{j,k}[s]$ defined in~\eqref{eq:noisy-projected-shell-coefficient}. Let $c(R_j)\in\mathbb C^{D_{\mathrm{proj}}}$ denote the vector obtained by stacking the corresponding noiseless coefficients over all retained planar frequencies and shells, where $D_{\mathrm{proj}}$ is the number of retained coefficients. The additive noise is taken to be circular complex Gaussian with variance $\sigma^2$ per coefficient. For the sampled rotations, we define
\begin{align}
    \mathrm{SNR} := \frac{ \frac{1}{n}\sum_{j=1}^{n}\|c(R_j)\|_2^2 }{ D_{\mathrm{proj}}\sigma^2 }.
    \label{eq:numerical-snr}
\end{align}
This is the empirical counterpart of the Haar-averaged signal-to-noise ratio.

Because the volume is identifiable only up to a global $\mathrm{O}(3)$ transformation, we first align the estimate using the degree-one block,
\begin{align}
    Q^\star := \argmin_{Q\in\mathrm{O}(3)} \|\widehat A_1-\rho_1(Q)A_1\|_F^2,
    \label{eq:numerical-global-alignment}
\end{align}
where $\rho_\ell(Q)$ denotes the action of $Q$ on the degree-$\ell$ real spherical-harmonic coefficients. We then report the normalized reconstruction error over all coefficients and, when assessing the forward frequency march separately, the corresponding error restricted to the isolating shells:
\begin{align}
    \operatorname{RMSE}(\widehat A,A)
    &:=
    \left( \frac{\sum_{\ell=0}^{L}\sum_{s\in[S_\ell]} \|\widehat A_\ell[s]-\rho_\ell(Q^\star)A_\ell[s]\|_2^2}{\sum_{\ell=0}^{L}\sum_{s\in[S_\ell]}\|A_\ell[s]\|_F^2} \right)^{1/2},
    \label{eq:numerical-nre}\\ \operatorname{RMSE}_{\mathrm{march}}(\widehat A,A)
    &:=    \left(\frac{\sum_{\ell=2}^{L}\sum_{s\in T_\ell^{\mathrm{iso}}} \|\widehat A_\ell[s]-\rho_\ell(Q^\star)A_\ell[s]\|_2^2} {\sum_{\ell=2}^{L}\sum_{s\in T_\ell^{\mathrm{iso}}}\|A_\ell[s]\|_F^2} \right)^{1/2}.
    \label{eq:numerical-march-nre}
\end{align}
Thus, $\operatorname{RMSE}$ measures the complete reconstruction error, whereas $\operatorname{RMSE}_{\mathrm{march}}$ measures the error on the coefficients recovered during the forward frequency march. The same global alignment $Q^\star$ is used in both metrics. The experiments were run on a MacBook Pro equipped with an Apple M3 Pro processor with 11 cores and 36 GB of memory. The code used to reproduce the experiments is available at~\url{https://github.com/krshay/projected-moments-recovery}.

\subsection{Validation of the constructive recovery}
\label{sec:constructive-recovery-experiments}

We first evaluate the constructive recovery procedure underlying Theorem~\ref{thm.shells}. We examine its behavior at the population level, its sensitivity to violations of the decay-by-three condition, and its stability when the projected moments are estimated from finitely many noisy observations.

\paragraph{Noiseless recovery and the shell condition.}
We begin with the TRPV1 and 80S volumes using harmonic bandlimit $L=9$ and $K=15$ radial shells on a $31^3$ grid. Since every harmonic degree is initially represented on all $15$ shells, the native shell counts are
\begin{align}
    (S_0,\ldots,S_9) = (15,15,15,15,15,15,15,15,15,15),
\end{align}
which do not satisfy Assumption~\ref{assum:decay}. In particular, $T_\ell^{\mathrm{iso}}$ is empty for $\ell\leq7$, so the forward frequency march cannot isolate the harmonic degrees as required by the constructive proof.

To test the theorem under its assumptions, we impose the decay-by-three support
\begin{align}
    (S_0,\ldots,S_9)  = (15,15,12,12,9,9,6,6,3,3).
    \label{eq:decay3}
\end{align}
We consider two versions of each molecular volume. In the first, the observable third moment $\mathcal{P}_3$ is computed from the original full-support coefficients while the recovery procedure assumes~\eqref{eq:decay3}. Thus, coefficients with $s>S_\ell$ contribute to the data but are omitted from the reconstruction model; these coefficients contain $7.5\%$ of the coefficient energy for TRPV1 and $4.4\%$ for the 80S ribosome. In the second, the coefficients outside~\eqref{eq:decay3} are set to zero before $\mathcal{P}_3$ is computed, so that the volume satisfies the assumptions of Theorem~\ref{thm.shells}. In both cases, $\mathcal{P}_3$ is evaluated analytically, and hence these experiments contain neither sampling error nor additive noise.

\begin{table}[t]
    \centering
    \begin{tabular}{llrrr}
        \toprule
        Volume & Coefficients & Forward march & All coefficients & Time (s) \\
        \midrule
        TRPV1 & full support ($7.5\%$ tail) & $1.56\times10^{2}$ & $2.26\times10^{7}$ & 8.8 \\
        TRPV1 & decay-by-three & $1.84\times10^{-12}$ & $6.66\times10^{-13}$ & 7.1 \\
        80S   & full support ($4.4\%$ tail) & $5.86\times10^{1}$ & $8.49\times10^{4}$ & 6.6 \\
        80S   & decay-by-three & $1.60\times10^{-13}$ & $3.26\times10^{-14}$ & 7.2 \\
        \bottomrule
    \end{tabular}
    \caption{\textbf{Noiseless constructive recovery.}
    Normalized coefficient errors from the analytic projected third moment.
    ``Forward march'' reports $\operatorname{RMSE}_{\mathrm{march}}$ on the isolating shells, while ``All coefficients'' reports the overall reconstruction $\operatorname{RMSE}$.}
    \label{tab:noiseless}
\end{table}

When the volume satisfies~\eqref{eq:decay3}, both structures are recovered to numerical precision. In contrast, even the relatively small model mismatch in the full-support volumes leads to large errors. The effect is particularly pronounced in the reverse backfill, which subtracts contributions formed from previously recovered higher-degree coefficients and therefore propagates errors accumulated during the forward march. For example, at $\ell=2$ the forward-marched TRPV1 shells have relative error $1.39\times10^{3}$, whereas the corresponding backfilled shells reach $2.07\times10^{8}$. This experiment illustrates the importance of the decay-by-three condition for the proposed constructive recovery procedure: it provides the structural mechanism used to isolate harmonic degrees.

\paragraph{Robustness to noisy moment estimates.}
We next study the finite-sample implementation under the assumptions of Theorem~\ref{thm.shells}. We generate a synthetic real volume with $L=5$ and
\begin{align}
    (S_0,\ldots,S_5) = (9,9,6,6,3,3),
    \label{eq:synthetic-shell-counts}
\end{align}
which satisfies Assumption~\ref{assum:decay}. 

We first fix $n=7.5\times10^5$ observations and vary the SNR between $0.1$ and $5$. For each SNR, we perform $10$ trials, independently drawing the viewing rotations and additive noise while keeping the underlying volume fixed. Within each trial, the same noiseless projections are used across SNR values and only the noise realization and variance are changed. Figure~\ref{fig:error-snr} shows the mean relative reconstruction error over the $10$ trials. The error decreases systematically with increasing SNR, demonstrating stability of the empirical frequency-marching procedure as the projected moments become more accurately estimated.

\paragraph{Dependence on the number of observations.}
Finally, we fix $\mathrm{SNR}=1$ and vary the number of observations from $2\times10^4$ to $5\times10^5$, again using $10$ independent trials for each value of $n$. Figure~\ref{fig:error-observations} reports the mean relative reconstruction error. The observed slope is approximately $-1/2$, consistent with the standard $n^{-1/2}$ estimation error of a fixed finite-dimensional population moment. 

\begin{figure}[t]
    \centering
    \subfigure[Robustness to noise.]{
        \includegraphics[width=0.47\linewidth]
        {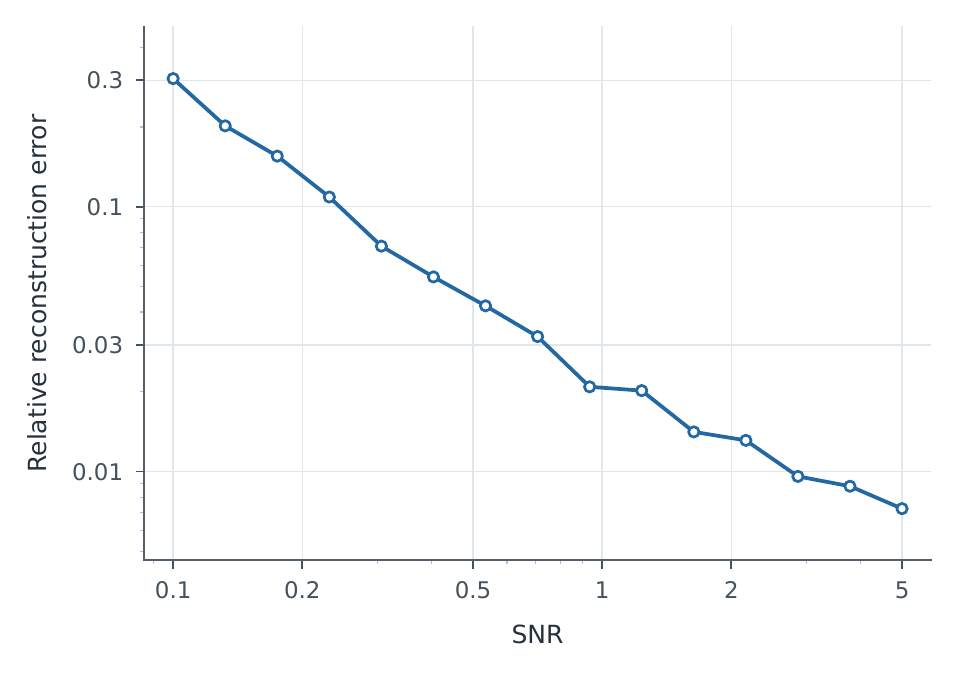}
        \label{fig:error-snr}
    }
    \hfill
    \subfigure[Sample complexity.]{
        \includegraphics[width=0.47\linewidth]
        {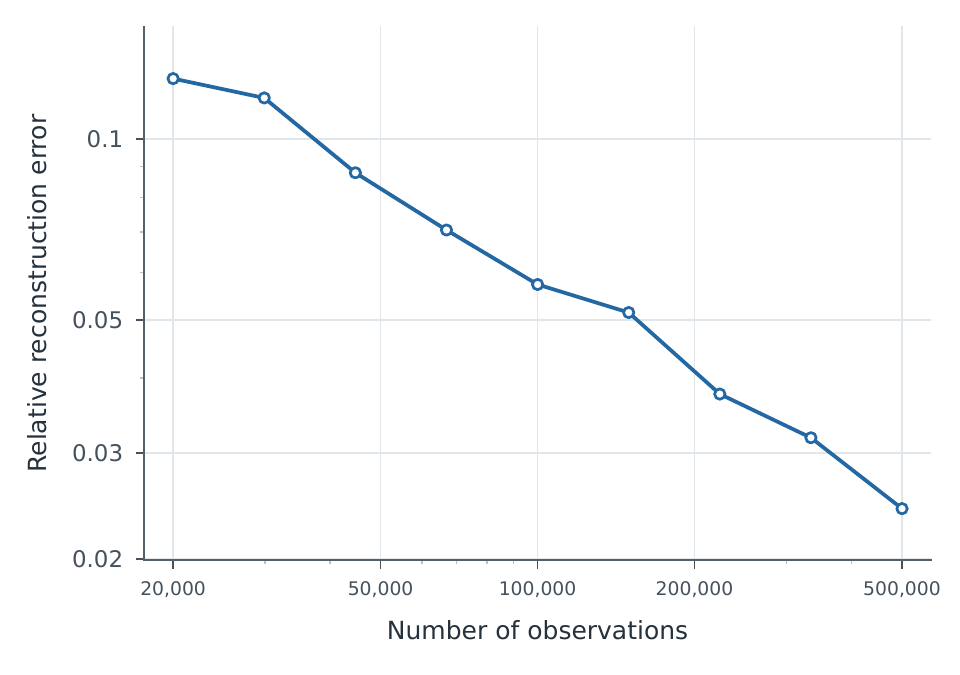}
        \label{fig:error-observations}
    }
    \caption{\textbf{Frequency-marching recovery from noisy observations.}
    Mean relative reconstruction error (as defined in~\eqref{eq:numerical-nre}), averaged over 10 trials:
    (a) versus SNR using $n=7.5\times10^5$ observations;
    (b) versus the number of observations $n$ at $\mathrm{SNR}=1$. 
    }
    \label{fig:frequency-marching-noise-samples}
\end{figure}

\subsection{Direct third-moment optimization}
\label{sec:direct-third-moment-optimization}

We conjecture that recovery from the projected third moment is still possible without the degree-dependent radial decay required by the constructive frequency-marching theorem. The following experiments provide numerical evidence for this conjecture. We set $L=5$ and retain the same number $K=12$ of radial shells at every harmonic degree; this shell count does not satisfy the assumption used by Algorithm~\ref{alg:population-recovery}. Given an estimate $\widehat{\mathcal{P}}_3$ of the projected third moment, formed from $n$ noisy projections as in Section~\ref{sec:numerical-setup}, we minimize
\begin{align}
    \mathcal L(A) =
    \frac{\|\mathcal P_3(A)-\widehat{\mathcal P}_3\|_{\mathrm F}^2} {\|\widehat{\mathcal P}_3\|_{\mathrm F}^2}
    \label{eq:direct-third-moment-loss}
\end{align}
over all coefficients, including those of degrees $\ell=0$ and $1$. We minimize~\eqref{eq:direct-third-moment-loss} by a quasi-Newton method, implemented using MATLAB’s \texttt{fminunc} solver. Every run terminated at a stationary point, with final gradient norm below $2.1\times10^{-8}$ (median $3\times10^{-9}$), within at most $3762$ iterations. Each run is initialized with i.i.d. Gaussian coefficients.

Since $\mathcal{P}_3$ is invariant under rotations and reflections of the volume, the minimizer is determined at best up to an element of
$\mathrm{O}(3)$. We therefore report the aligned error
\begin{align}
    \operatorname{err}(\widehat{A}) =
    \min_{g\in\mathrm{O}(3)}
    \frac{\|\widehat{A}-g\cdot A^\star\|_{F}}{\|A^\star\|_{ F}},
    \label{eq:aligned-error}
\end{align}
taken over all degrees $0\leq\ell\leq L$, since every coefficient is estimated.

We reconstruct the 80S ribosome (EMD-2660) and TRPV1 (EMD-24084, a reconstruction without imposed symmetry) from noisy projections, sweeping the SNR from high to low with $n=10^5$ observations and $20$ random initializations per SNR; see Figure~\ref{fig:80S_TRPV1_vs_noise}. At $\mathrm{SNR}=\infty$ the noise variance is zero but the moment is still estimated from $n$ projections, so the error is set entirely by sampling; because each SNR uses an independent draw of viewing directions, this sampling error differs between SNR values, which is why the TRPV1 error at $\mathrm{SNR}=\infty$ lies slightly above that at $\mathrm{SNR}=30$.

The restart experiments also suggest favorable optimization behavior for the tested volumes. For the 80S ribosome, at every $\mathrm{SNR}\geq 1$, the aligned errors of all $20$ random initializations agree to within $10^{-6}$. Moreover, at $\mathrm{SNR}=\infty$ and $\mathrm{SNR}=1$, the recovered coefficients coincide to within $3\times10^{-5}$ after alignment, indicating that all initializations converge to essentially the same solution. For TRPV1 the restarts reach several distinct stationary points of nearly equal loss, all at a similar distance from the truth: the standard deviation of the aligned error across restarts is at most $6.2\times10^{-3}$ at every SNR. In neither case did any run converge to a solution far from the truth. These experiments do not prove identifiability or the absence of spurious minima, but they are consistent with both. A fairly benign optimization landscape has also been observed for the non-uniform second moment~\cite{sharon2020method}. As a final demonstration for the 80S ribosome, Figure~\ref{fig:reconstruction} shows the recovered volume next to the true volume for $L=13$, $\mathrm{SNR}=2$ and $n=2\times10^5$ observations, rendered with ChimeraX~\cite{meng2023ucsf}. 

\begin{figure}[t]
    \centering
    \subfigure[Robustness to noise.]{
        \includegraphics[width=0.47\linewidth]        {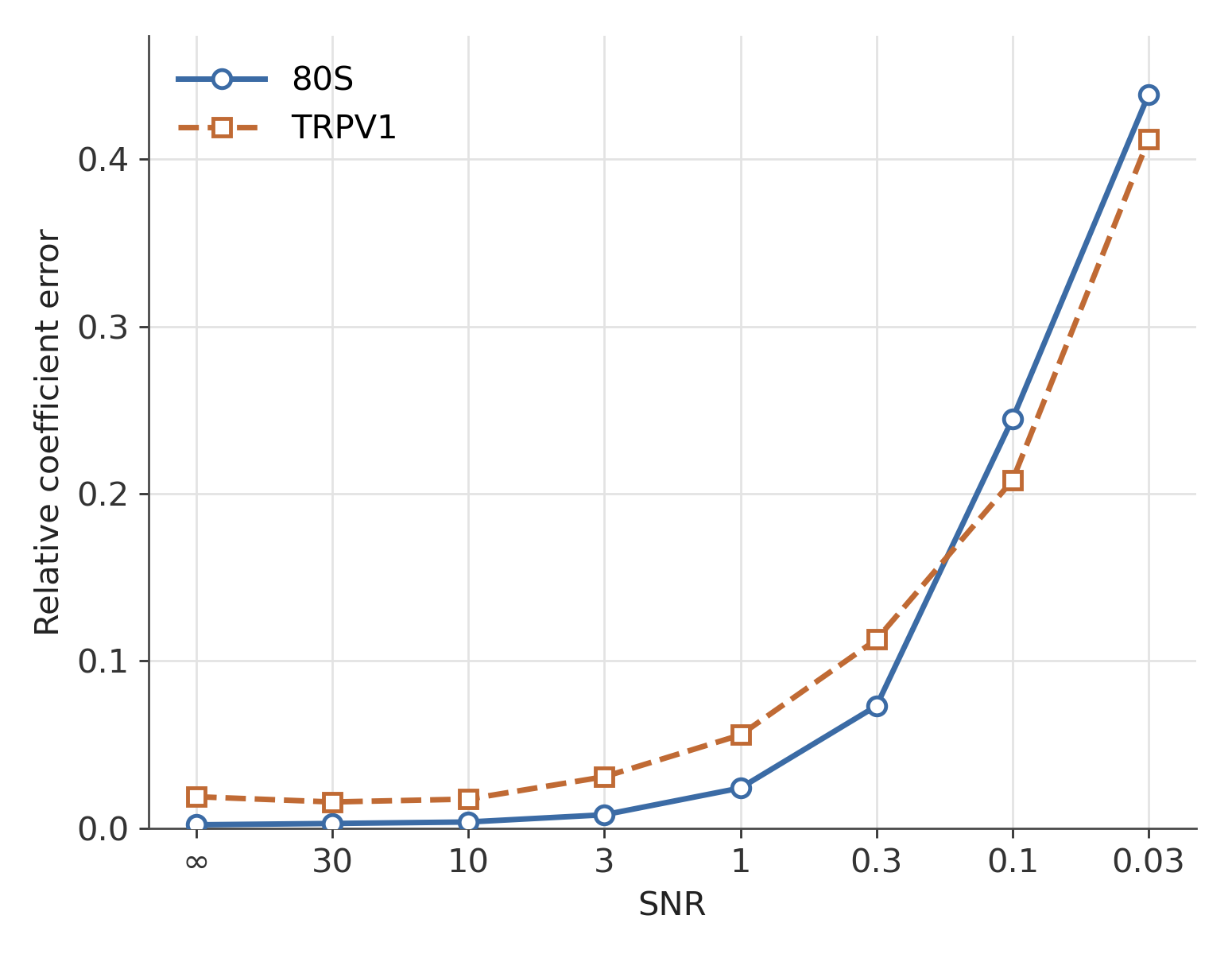}
        \label{fig:80S_TRPV1_vs_noise}
    }
    \hfill
    \subfigure[80S reconstruction.]{
        \raisebox{6mm}[\height][\depth]{%
            \includegraphics[width=0.47\linewidth]
            {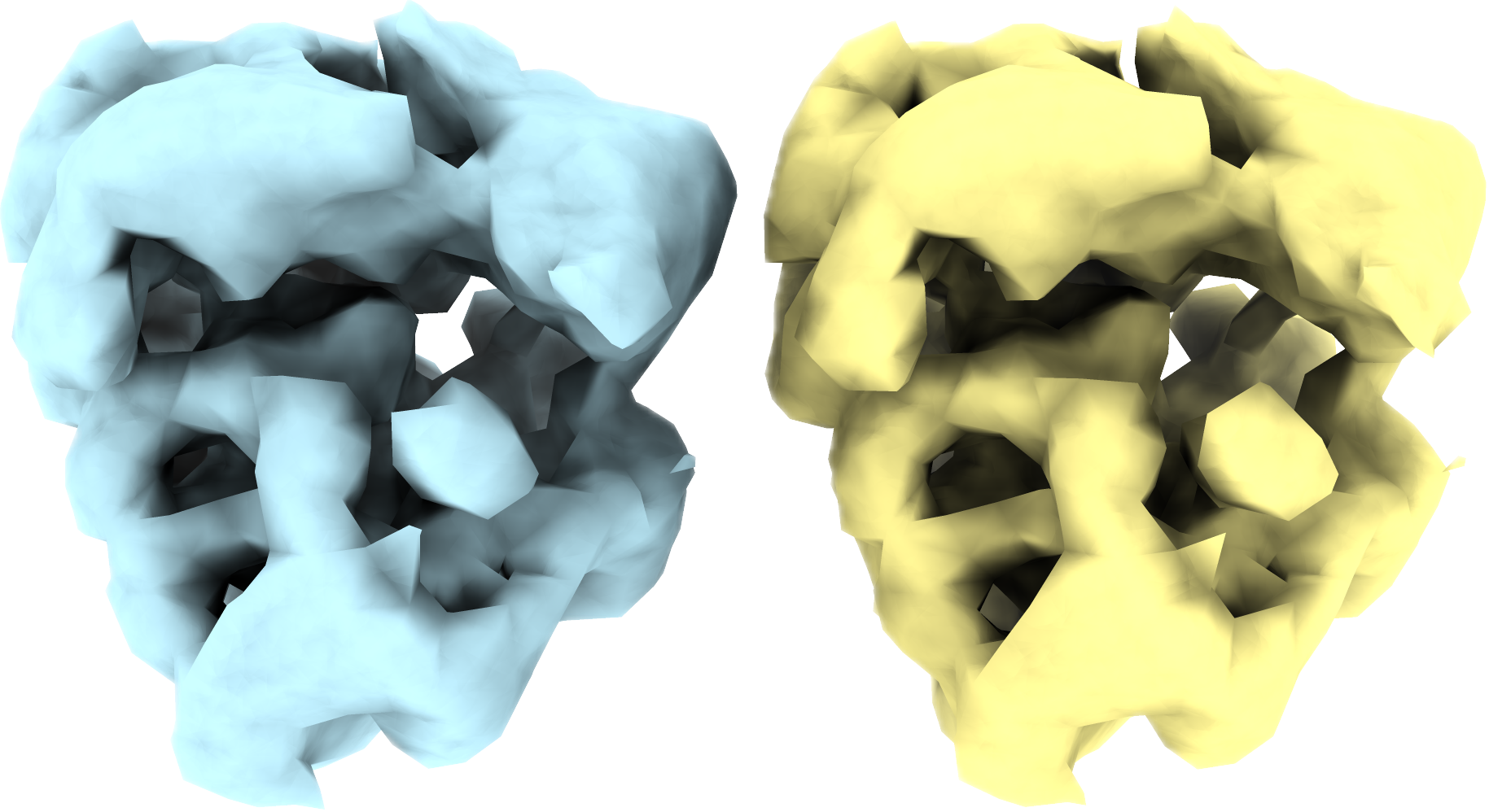}%
        }
        \label{fig:reconstruction}
    }
    \caption{\textbf{Direct third-moment optimization.}
    (a) Mean aligned relative coefficient error~\eqref{eq:aligned-error} over all degrees for 80S and TRPV1 versus SNR, using $n=10^5$ observations and averaged over $20$ random initializations per SNR. At $\mathrm{SNR}=\infty$ the noise variance is zero but the moment is still estimated from a finite number of $n$ projections.
    (b) Isosurface renderings of the true 80S volume (left) and the volume recovered at $\mathrm{SNR}=2$ using $2\times10^5$ observations (right), with $L=13$.}
    \label{fig:direct-optimization-results}
\end{figure}

\section{Discussion and outlook}
\label{sec:discussion}

We have shown that, in the spherical-shell model with uniformly distributed viewing rotations, the first three projected moments generically determine a real volume up to its global $\mathrm{O}(3)$ ambiguity under Assumption~\ref{assum:decay}. The numerical experiments confirm the constructive recovery mechanism within the assumed shell model and also show that direct third-moment optimization can succeed when the decay-by-three condition is violated, suggesting that this condition may be sufficient for the present proof rather than necessary for identifiability.

\paragraph{Future work.}
Several directions remain open. First, we aim to extend the identifiability analysis beyond the spherical-shell representation to more general radial bases, such as spherical-Bessel and spherical-prolate bases used in cryo-EM. Second, even within the spherical-shell model, it remains unclear whether the shell-separation assumption is necessary for generic identifiability or merely for the proposed constructive recovery procedure. Third, independently of identifiability, an important question is to characterize the numerical conditioning of the inverse problem and its dependence on the distribution of energy across harmonic degrees, radial support, and the choice of representation. Such an analysis would help distinguish fundamental limitations of moment-based recovery from instability arising from particular reconstruction algorithms.

An additional direction is to incorporate more realistic cryo-EM forward models. Contrast transfer functions can be included as known frequency-dependent operators in the projected-moment map. Other practically important extensions include nonuniform viewing distributions, unknown in-plane translations, heterogeneous mixtures of volumes, and symmetry-constrained molecules; see, for example,~\cite{bendory2020single,singer2020computational}. Each changes the structure of the observable moment tensors and may require different coupling or separation arguments.

\section*{Acknowledgments}
Portions of the code used for the numerical experiments were developed with the assistance of Claude (Anthropic) and ChatGPT (OpenAI), which were used to write and debug implementation code. All algorithms, mathematical derivations, and experimental designs are the authors' own. The authors reviewed, tested, and validated all AI-assisted code and take full responsibility for its correctness.

 T.B. and N.S. are supported in part by NSF-BSF under grant 2024791. T.B. is also supported in part by ISF under grant 1046/26. N.S. is also supported partially by the BSF award 2024266 and the DFG award 514588180.
 
\bibliographystyle{plain}

\end{document}